\documentclass[12pt, reqno]{amsart}
\usepackage{amsmath, amsthm, amscd, amsfonts, latexsym, amssymb, graphicx, color}
\usepackage[bookmarksnumbered, colorlinks, plainpages]{hyperref}

\newtheorem{thm}{Theorem}[section]
\newtheorem{prop}[thm]{Proposition}
\newtheorem{cor}[thm]{Corollary}
\theoremstyle{definition}
\newtheorem{defn}[thm]{Definition}
\newtheorem{lem}[thm]{Lemma}
\newtheorem{note}[thm]{Note}
\newtheorem{exam}[thm]{Example}

\theoremstyle{remark}

\numberwithin{equation}{section}

\newcommand{\N}{\Bbb{N}}

\newcommand{\Ann}{{\rm{Ann}}}

\newcommand{\rad}{{\rm{rad}}}

\begin{document}

\makeatletter \oddsidemargin.9375in \evensidemargin \oddsidemargin
\marginparwidth1.9375in \makeatother

\setcounter{page}{1}
\title[Sa-small submodules w.r.t. an arbitrary submodule]{Semiannihilator small submodules with respect to an arbitrary submodule}
\author[S. Rajaee]{S. Rajaee $^{*}$}
\author[F. Farzalipour]{F. Farzalipour}
\author[R. Poyan]{M. Poyan}
\thanks{{\scriptsize
\hskip -0.4 true cm MSC(2010): Primary:13C13, 13C99, 16D80.
$*$Corresponding author}}
\maketitle
\begin{abstract}
In this paper we introduce a new concept namely $T$-semiannihilator small ($T$-sa-small for short) submodules of an $R$-module $M$ with respect to an arbitrary submodule $T$ of $M$ which is a generalization of the concept of semiannihilator small submodules. We say that a submodule $N$ of $M$ is $T$-sa-small in $M$ provided for each submodule $X$ of $M$ such that
$T\subseteq N+X$ implies that ${\rm Ann}(X)\ll (T:M)$.  We investigate some new results concerning to this class of submodules. Among various results we prove that for a prime module $M$ over a semisimple ring $R$, if $N$ is a sa-small submodule of $M$, then for every submodule $T$ of $M$ such that $N\subsetneq T$, $N$ is also a $T$-sa-small submodule of $M$. For a faithful finitely generated multiplication module $M$, we prove that $N$ is a $T$-sa-small submodule of $M$ if and only if  $(N:M)$ is a $(T:M)$-sa-small ideal of $R$. 
\end{abstract}
\section{Introduction} 
Throughout this paper, $R$ will denote a commutative ring with identity $1\neq 0$ and $M$ a nonzero unital $R$-module. We use the notations $\subseteq $ and $\leq $ to denote inclusion and submodule. A nonempty subset $S$ of $R$ is said to be a {\it multiplicatively closed set} (briefly, m.c.s) of $R$ if $0\notin S$, $1\in S$ and $st\in S$ for each $s, t\in S$. For an $R$-module $M$, the set of all submodules of $M$, denoted by ${\rm L}(M)$ and also ${\rm L}^{*}(M)={\rm L}(M)\setminus \{M\}$.
As usual, the rings of integers and integers modulo $n$ will be denoted by $\mathbb{Z}$ and $\mathbb{Z}_{n}$, respectively. A module $M$ on a ring $R$ (not necessarily commutative) is called {\it prime} if for every nonzero submodule
$K$ of $M$, ${\rm Ann}(K) ={\rm Ann}(M)$. An $R$-module $M$ is called a {\it multiplication module},
if every submodule $N$ of $M$ has the form $N = IM$ for some ideal $I$ of $R$, and in
this case, $N =(N:_{R}M)M$, see \cite{Bar}. A submodule $N$ of $M$ is called {\it small} ({\it superfluous}) which is denoted by $N\ll M$, if for every submodule $L$ of $M$, $N+L = M$, implies that $L = M$. Clearly, the zero submodule of every nonzero module is superfluous. 
The {\it Jacobson radical} of $M$, denoted by ${\rm J}(M)$ is the intersection of all maximal submodules of $M$ and also it is the sum of all small submodules of $M$, i.e.,
${\rm J}(M)=\cap _{\mathfrak{M}\in {\rm Max}(M)}\mathfrak{M}=\sum _{N\ll M}N$. If $M$ does not have maximal submodules, we put ${\rm J}(M)=M$. Consequently, if ${\rm J}(M)$ is a small submodule of $M$, then ${\rm J}(M)$ is the largest small submodule of $M$. If $M$ is a finitely generated nonzero module, then $M\neq {\rm J}(M)$ and ${\rm J}(M)$ is the largest small submodule of $M$. We refer the reader to \cite{AnFu, AT} for the basic properties and more information on small submodules. We know that if  $M$ is a semisimple module, then the zero submodule is the only
small submodule of $M$ and $M$ is the only essential submodule of $M$. 

Gilmer \cite[p.60]{Gil} has defined the concept of {\it cancellation ideal} to be the
ideal $I$ of $R$ which satisfies the following: Whenever $AI = BI$ with $A$ and $B$ are ideals of $R$ implies $A = B$. Mijbass in \cite{Mij} has generalized this concept to
modules. An $R$-module $M$ is called a {\it cancellation module} whenever $IM = JM$
with $I$ and $J$ are ideals of $R$ implies $I = J$.

We recall that $R$ is a {\it von Neumann regular ring} (associative, with 1, not necessarily commutative) if for every element $a$ of $R$, there is an element $b\in R$ with $a = aba$. These rings are characterized by the fact that every left $R$-module is flat.
\section{Preliminaries and Notations}
In \cite{AK}, the authors introduced the concept of {\it annihilator-small submodules} of any right $R$-module $M$. For an unitary right $R$-module $M$ on an associative ring $R$ with identity they called a submodule $K$ of $M_{R}$ {\rm annihilator-small} if $K + T = M$, such that $T$ is a submodule of $M_{R}$, implies that $\ell _{S}(T) = 0$, where $\ell _{S}(T)$ indicates the left
annihilator of $T$ over $S={\rm End}(M_{R})$.

A submodule $N$ of $M$ is called an {\em $R$-annihilator-small} (breifly {\em R-a-small}) {\em submodule} of $M$, if $N+X = M$ for some submodule $X$ of $M$, implies that ${\rm Ann}(X) =0$, see \cite{AL}. We use the notation $N\ll ^{a}M$ to indicate this concept. 
In \cite{Yas}, the author introduced the concept of a semiannihilator small submodule $N$ of a module $M$  on a commutative ring $R$ with identity $1\neq 0$ such that $N$ is called {\it semiannihilator small} ({\it sa-small} for short), denoted by $N\ll ^{sa}M$, if for every submodule $L$ of $M$ with $N+L=M$ implies that ${\rm Ann}(L)\ll R$. It is clear that every R-a-small
submodule is sa-small, but the converse is not true. Also a sa-small submodule need not be small, for example consider $M=\mathbb{Z}$ as a $\mathbb{Z}$-module, every proper submodule of $\mathbb{Z}$ is a sa-small submodule whereas the zero submodule of $\mathbb{Z}$ is the only small submodule of $\mathbb{Z}$.  Let $M$ be a faithful $R$-module, then clearly every small submodule is an R-a-small submodule and hence it is a sa-small submodule. The converse is not true, because consider the $\mathbb{Z}_{8}$-module  $M=\mathbb{Z}_8$, then $\mathbb{Z}_8=\mathbb{Z}_8+\langle \overline{2}\rangle $ and ${\rm Ann}(\langle \overline{2}\rangle )=\langle \overline{4}\rangle \neq \langle \overline{0}\rangle $, whereas ${\rm Ann}(\langle \overline{2}\rangle )\ll \mathbb{Z}_8$.

Similarly, an ideal $I$ is a sa-small ideal of a ring $R$, if it is a sa-small submodule of $R$ as an $R$-module, see \cite[Definition 2.1]{Yas}. We know that an ideal $I$ of $R$ is small in $R$ if and only if $I\subseteq {\rm J}(R )$. Therefore a submodule $N$ of a module $M$ is sa-small in $M$ if $N+X=M$ for a submodule $X$ of $M$, implies that ${\rm Ann}(X)\ll {\rm J}(R)$, see \cite[Proposition 2.3]{Yas}. A non-trivial $R$-module $M$ is called a {\it semiannihilator hollow} ({\it sa-hollow} for short) module if every proper submodule of $M$ is sa-small in $M$, see \cite[Definition 3.1]{Yas}. An $R$-epimorphism $f:M\rightarrow N$ is called a {\it sa-small epimorphism} whenever ${\rm Ker}(f)\ll ^{sa}M$. 

An $R$-module $M$ is said to be a {\it comultiplication module} if
for every submodule $N$ of $M$ there exists an ideal $I$ of $R$ such that
$N ={\rm Ann}_{M}(I)$. An $R$-module $M$ satisfies the {\it double annihilator condition} (DAC
for short), if for each ideal $I$ of $R$, $I ={\rm Ann}_{R}({\rm Ann}_{M}(I))$. Also $M$
is said to be a {\it strong comultiplication module}, if $M$ is a comultiplication module which satisfies DAC, see \cite{AF}.
\section{Sa-small submodules w.r.t. an arbitrary submodule}
In this section we generalize the concept of sa-small submodules to the $T$-sa-small submodules of $M$ with respect to an arbitrary submodule $T$ of $M$. Moreover, we investigate some new other
properties of the sa-small submodules of an $R$-module $M$ and we will generalize these properties to this new class of submodules of $M$. 
\begin{defn}
Let $M$ be an $R$-module and let $T$ be an arbitrary submodule of $M$.
\begin{itemize}
\item[(i)] We say that a submodule $N$ of $M$ is a {\it $T$-semiannihilator small} (berifly, $T$-sa-small) submodule of $M$, denoted by $N\ll _{T}^{sa}M$, provided for every nonzero submodule $X\leq M$ such that $T\subseteq N+X$ implies that ${\rm Ann}(X)\ll (T:_{R}M)={\rm Ann}_{R}(M/T)$. Equivalently, if for a submodule $X$ of $M$, ${\rm Ann}(X)$ is not small in $(T:_{R}M)$, then $T\nsubseteq N+X$. 
\item[(ii)] We say that $M$ is a {\it $T$-sa-small hollow module} if every submodule $N$ of $M$ is a $T$-sa-small submodule of $M$. In particular, for an arbitrary ideal $A$ of $R$, we say that an ideal $I$ of $R$ is an {\it $A$-sa-small ideal} of $R$ if $I$ is an $A$-sa-small submodule of $R$ as an $R$-module. We shall denote the
sum of all $T$-sa-small submodules of $M$ by ${\rm J}_{T}^{sa}(M)$.
\item[(iii)] Let $f:M\rightarrow N$ be an $R$-epimorphism and let $T$ be an arbitrary submodule of $M$, we say that $f$ is a {\it $T$-sa-small epimorphism} in case ${\rm Ker}(f)\ll _{T}^{sa}M$.
\end{itemize}
We denote the set of all small (resp. sa-small, $T$-sa-small) submodules of $M$ by ${\rm S}(M)$ 
(resp. ${\rm S}^{sa}(M)$, ${\rm S}_{T}^{sa}(M)$).
\end{defn}
\begin{note}
Let $M$ be an $R$-module and let $T$ be an arbitrary submodule of $M$.
\begin{itemize}
\item[(i)] If we take $T=M$, then the notions of $T$-sa-small submodules and sa-small submodules are equal.
\item[(ii)] Let $T=0$, then for submodules $N, X$ of $M$, since $0\subseteq N+X$ hence $N\ll _{0}^{sa}M$ implies that ${\rm Ann}(X)\ll {\rm Ann}(M)$. Now since ${\rm Ann}(M)\subseteq {\rm Ann}(X)$ therefore ${\rm Ann}(X)={\rm Ann}(M)$. This is impossible because a nonzero module $M$ is never small in itself. It concludes that a nonzero module $M$ has no any $0$-sa-small submodule. 
\item[(iii)] If $T\neq 0$, then $N\ll ^{sa}_{T}M$ implies that $T\nsubseteq N$, otherwise, $T\subseteq N+0$ conclude that 
$R={\rm Ann}(0)\ll (T:M)$ which is impossible. Let $M$ be a finitely generated
$R$-module and let $T$ be a nonzero proper arbitrary submodule of $M$ then by Zorn's lemma there exists a maximal submodule $N$ of $M$ such that $T\subseteq N$. This implies that $N$ is not a $T$-sa-small submodule of $M$.
\end{itemize}
\end{note}
\begin{thm}\label{t.2.3}
Let $M$ be an $R$-module and let $T$ be a submodule of $M$. The following assertions hold.  
\begin{itemize}
\item[(i)] Every $T$-sa-small submodule of $M$ is a sa-small submodule of $M$.
\item[(ii)] If $M$ is a faithful prime $R$-module, then $M$ is a $T$-sa-small hollow module.
\item[(iii)] If $M$ is a prime module on a semisimple ring $R$ and $N\ll ^{sa}M$, then for every submodule $T\varsupsetneq N$ of $M$, $N\ll _{T}^{sa}M$.
\end{itemize}
\end{thm}
\begin{proof}
(i) Assume that $N\ll _{T}^{sa}M$ and $N+K=M$ for some module $K$ of $M$. Since $T\subseteq N+K$ and $N\ll ^{sa}_{T}M$ hence ${\rm Ann}(K)\ll (T:M)\leq R$. By virtue of \cite[Remark 2.8, (2)]{AT}, ${\rm Ann}(K)\ll R$ this implies that $N\ll ^{sa}M$.\\
(ii) By hypothesis, $T\subseteq N+X$ implies that $0={\rm Ann}(X)={\rm Ann}(M)$ which is small in $(T:M)$.\\
(iii) By \cite[Proposition 3.5]{AT}, $M$ is a semisimple module, hence $M=T+T'$ for some submodule $T'$ of $M$. Suppose that $N\ll ^{sa}M$ and $T\subseteq N+K$ for some submodule $K$ of $M$. This implies that $M=T+T'\subseteq N+K+T'$ and so $N+K+T'=M$. Since $N\ll ^{sa}M$ hence ${\rm Ann}(K+T')\ll R$. Since $M$ is prime hence
\[{\rm Ann}(K+T')={\rm Ann}(K)\subseteq (T:M)\subseteq (T:N)=R.\]
By virtue of \cite[Remark 2.8, (3)]{AT}, ${\rm Ann}(K)\ll (T:M)$ since $R$ is semisimple and the proof is complete.
\end{proof}
The following example shows that in general the concepts of small submodules and sa-small submodules are indepentent of each other.  In (iii) we show that the converse of Theorem \ref{t.2.3} is not true in general and also in (iv) we show that for $R$-modules $M, M'$, if $f:M\rightarrow M'$ is an $R$-epimorphism, then the image of a $T$-sa-small submodule of $M$ need not be an $f(T)$-sa-small submodule in $M'$. 
\begin{exam}
Let $\mathbb{Z}$ and  $\mathbb{Z}_{n}$ be the rings of integers and integers modulo $n$, respectively. 
\begin{itemize}
\item[(i)] Consider $M=\mathbb{Z}_{6}$ as a $\mathbb{Z}$-module. Since ${\rm Ann}(\mathbb{Z}_{6})=6\mathbb{Z}$
is not a small ideal of $\mathbb{Z}$, hence $\langle \overline{0}\rangle \notin {\rm S}^{sa}(\mathbb{Z}_{6})$ whereas 
$\langle \overline{0}\rangle \in {\rm S}(\mathbb{Z}_{6})$.  We note that the only nonzero proper submodules of $\mathbb{Z}_{6}$ are $N=\langle \overline{2}\rangle $ and $L=\langle \overline{3}\rangle $ where $N+L=\mathbb{Z}_{6}$ and both of ${\rm Ann}(N)=3\mathbb{Z}$ and ${\rm Ann}(L)=2\mathbb{Z}$ are not small ideals of $\mathbb{Z}$. It concludes that  ${\rm S}^{sa}(\mathbb{Z}_{6})=\emptyset $ whereas ${\rm S}(\mathbb{Z}_{6})=\{\langle \overline{0}\rangle \}$.
\item[(ii)] Take $M=\mathbb{Z}$ as a $\mathbb{Z}$-module. We know that $k\mathbb{Z}+s\mathbb{Z}=\mathbb{Z}$ if and only if $(k, s)=1$. Since for every submodule $s\mathbb{Z}$ of $\mathbb{Z}$, ${\rm Ann}_{\mathbb{Z}}(\mathbb{Z})={\rm Ann}_{\mathbb{Z}}(s\mathbb{Z})=0\ll \mathbb{Z}$ hence $k\mathbb{Z}\ll ^{sa}\mathbb{Z}$ for every proper submodule $k\mathbb{Z}$ of $\mathbb{Z}$. It concludes that ${\rm S}^{sa}(\mathbb{Z})={\rm L}^{*}(\mathbb{Z})$ whereas 
 ${\rm S}(\mathbb{Z})=\{0\}$.  
\item[(iii)] 
We take the $\mathbb{Z}$-module $M=2\mathbb{Z}\times \mathbb{Z}_{8}$. Then $N=\langle (0,\bar{0})\rangle $ is a sa-small submodule of $M$ but $N$ is not a $T$-sa-small submodule of $M$ for submodule $T=\langle (0,\bar{4})\rangle $ of $M$ since $T\subseteq N+\langle (0,\bar{2})\rangle $ whereas ${\rm Ann}(\langle (0,\bar{2})\rangle )=4\mathbb{Z}$ is not small in $(T:M)=0$.
\item[(iv)] Consider the natural $\mathbb{Z}$-epimorphism $\pi :\mathbb{Z}\longrightarrow \mathbb{Z}_{8}$ where $\pi (n)= \overline{n}$. Take $N=0$ and $T=2\mathbb{Z}$. Clearly $0\ll _{2\mathbb{Z}}^{sa}\mathbb{Z}$, because we have
$2\mathbb{Z}\subseteq 0+2\mathbb{Z}$ and also $2\mathbb{Z}\subseteq 0+\mathbb{Z}$ and then
\[0={\rm Ann}(2\mathbb{Z})\ll (2\mathbb{Z}:_{\mathbb{Z}}\mathbb{Z})=2\mathbb{Z},\]
\[0={\rm Ann}(\mathbb{Z})\ll (2\mathbb{Z}:_{\mathbb{Z}}\mathbb{Z})=2\mathbb{Z}.\]
But $\pi (N)=\pi (0)=\langle \bar{0}\rangle $ is not $\pi (T)$-sa-small submodule of  $\mathbb{Z}_{8}$ since $\pi (T)=\langle \bar{2}\rangle $ and $\langle \bar{2}\rangle \subseteq \langle \bar{0}\rangle+\mathbb{Z}_{8}$ whereas
${\rm Ann}(\mathbb{Z}_{8})=8\mathbb{Z}$ is not small in $(\langle \bar{2}\rangle:\mathbb{Z}_{8})=2\mathbb{Z}$. 
\end{itemize}
\end{exam}
\begin{note}
Let $M$ be an $R$-module and let $T$ be an arbitrary submodule of $M$. For every submodule $K$ of $M$, $K\ll ^{sa}_{T}M$ if and only if $R$-epimorphism $p_{_{K}}:M\rightarrow M/K$ is a $T$-sa-small epimorphism. 
\end{note}
A submodule $N$ of an $R$-module $M$ is said to be {\it completely irreducible} if
$N=\cap _{i\in I}N_{i}$ where $\{N_{i}\}_{i\in I}$ is a family of submodules of $M$, then $N = N_{i}$ for
some $i\in I$. It is easy to see that every submodule of $M$ is an intersection of
completely irreducible submodules of $M$.
\begin{thm}\label{t2.5}
Let $M$ be an $R$-module and $N\leq M$.
\begin{itemize}
\item[(i)] Let $T\lneq T'\lneq M$. If $N\ll ^{sa}_{T}M$, then $N\ll ^{sa}_{T'}M$. 
\item[(ii)] If $T=T_{1}\,\cap \cdots \,\cap T_{s}$ and $N\ll ^{sa}_{T}M$, then $N\ll ^{sa}_{T_{i}}M$ for any $1\leq i\leq s$. Conversely, if $T$ is a completely irreducible submodule and for any $1\leq i\leq s$, $N\ll ^{sa}_{T_{i}}M$, then $N\ll ^{sa}_{T}M$.
\item[(iii)] If $T\leq K$ and $N\ll ^{sa}_{T}M$, then $N\ll ^{sa}_{T}K$.  
\item[(iv)] If $T=T_{1}+\cdots +T_{n}$ for some submodules $T_{i}$ of $M$ and $N\leq_{T_{i}}^{sa}M$ for some $1\leq i\leq n$, then $N\leq_{T}^{sa}M$.
\end{itemize}
\end{thm}
\begin{proof}
(i) Assume that $T'\subseteq N+X$ for some submodule $X$ of $M$, then $T\subseteq N+X$ and so ${\rm Ann}(X)\ll (T:M)\leq (T':M)$. By \cite[Remark 2.8, (2)]{AT}, ${\rm Ann}(X)\ll (T':M)$ as we needed.\\
(ii) It is clear by (i). Conversely, since  $T$ is a completely irreducible submodule hence there exists $1\leq i\leq s$ such that 
$T_{i}=T$ and the proof is complete again using by (i). \\
(iii) Assume that $T\subseteq N+L$ for some submodule $L$ of $K$. Then ${\rm Ann}(L)\ll (T:M)\leq (T:K)$ and this implies that ${\rm Ann}(L)\ll (T:K)$ and therefore $N\ll ^{sa}_{T}K$.\\
(iv) The proof is straightforward by (i).
\end{proof}
\begin{thm}
The following assertions hold.
\begin{itemize}
\item[(i)] ${\rm S}(R)\subseteq {\rm S}^{sa}(R)$.
\item[(ii)] If $R$ is an Artinian ring and $I\ll _{{\rm J}(R)}^{sa}R$, then for every maximal ideal $\mathfrak{m}$ of $R$,  $I\ll _{\mathfrak{m}}^{sa}R$.
\item[(iii)] Let $M=N\oplus K$ be a multiplication module such that $N, K$ are finitely generated submodules of $M$. Then $N$ and $K$ are not sa-small in $M$. 
\end{itemize}
\begin{proof}
(i) Let $I\in {\rm S}(R)$ and $I+J=R$ for some ideal $J$ of $R$. Then $J=R$ and hence ${\rm Ann}(J)={\rm Ann}(R)=0$ which is a small ideal of $R$. Therefore $I\in {\rm S}^{sa}(R)$.\\
(ii) Assume that ${\rm Max}(R)=\{\mathfrak{m}_{1}, \cdots , \mathfrak{m}_{s}\}$, then ${\rm J}(R)=\cap _{i=1}^{s}\mathfrak{m}_{i}$. The proof is clear by Theorem \ref{t2.5} (i).\\
(iii) By \cite[Corrollary 2.3]{Bast}, ${\rm Ann}(N)+{\rm Ann}(K)=R$. If $N\ll ^{sa}M$, then since $M=N+K$ hence ${\rm Ann}(K)\ll R$ and so ${\rm Ann}(N)=R$. It concludes that $N=0$ which is impossible.
\end{proof}
\end{thm}
\begin{thm}
Let $M$ be an $R$-module. The following assertions hold.
\begin{itemize}
\item[(i)] $M\notin {\rm S}^{sa}(M)$.
\item[(ii)] $0\in {\rm S}^{sa}(M)$ if and only if ${\rm Ann}(M)\ll R$.
\item[(iii)] If $R$ is a simple ring, then ${\rm L}^{*}(M)={\rm S}^{sa}(M)$.
\item[(iv)] Let $\mathfrak{m}\in {\rm Max}(R)\cap {\rm S}^{sa}(R)$. If $x\notin \mathfrak{m}$, then 
${\rm Ann}(Rx)\subseteq {\rm J}(R)$.
\item[(v)]  If $M$ is a prime module with ${\rm Ann}(M)\ll R$, then ${\rm L}^{*}(M)={\rm S}^{sa}(M)$. Moreover, in this case, ${\rm S}({M})\subseteq {\rm S}^{sa}(M)$.
\end{itemize}
\end{thm}
\begin{proof}
(i) Assume that $M\in {\rm S}^{sa}(M)$, then the equality $M+0=M$, implies that ${\rm Ann}(0)=R$ is a small ideal of $R$ which is impossible.\\
(ii) The proof is straightforward.\\
(iii) Let $R$ be a simple ring and let $N$ be a proper submodule of $M$ such that $N+X=M$ for some submodule $X$ of $M$. If ${\rm Ann}(X)=0$, then ${\rm Ann}(X)\ll R$ and the proof is complete. If ${\rm Ann}(X)=R$, then $X=0$ and hence $N=M$ which is contradiction.\\
(iv) Suppose that $x\in R-\mathfrak{m}$, then $\mathfrak{m}+Rx=R$ hence ${\rm Ann}(Rx)\ll R$ since $\mathfrak{m}\in {\rm S}^{sa}(R)$. This implies that ${\rm Ann}(Rx)\subseteq {\rm J}(R)$.\\
(v) It is clear.
\end{proof}
\begin{cor}
Every proper submodule of a faithful prime module $M$ is a sa-small submodule of $M$. 
\end{cor}
\begin{thm}\label{t2.3}
Let $R$ be a commutative ring. The following statments are true.
\begin{itemize}
\item[(i)] If $R$ has a nonzero comaximal sa-small ideal, then $R$ is not semisimple.
\item[(ii)] Every sa-small hollow semisimple ring is simple.
\item[(iii)] Let $R$ be a sa-small hollow ring. If $x\in {\rm Z}(R)$, then $R\neq Rx +Ry$ for some element $y\in R$ and also $1$ is the only nonzero idempotent element of $R$.
\item[(iv)] If $R$ is a von Neumann regular ring, then none of the finitely generated ideals of $R$ is a sa-small ideal of $R$.
\item[(v)] Let $R$ be an integral domain and let $M$ be a faithful multiplication module, then ${\rm L}^{*}(M)={\rm S}^{sa}(M)$. 
\end{itemize}
\end{thm}
\begin{proof}
(i) Suppose that $I$ is a nonzero comaximal sa-small ideal of $R$, then there exists an ideal $J$ of $R$ with $I+J=R$. We claim that $I\cap J\neq 0$. Assume $I\cap J=0$ and so $IJ=0$. Hence $R=I+J\subseteq {\rm Ann}(J)+{\rm Ann}(I)$, then ${\rm Ann}(I)+{\rm Ann}(J)=R$. Now $I\ll ^{sa}R$, implies that ${\rm Ann}(J)\ll R$ and thus ${\rm Ann}(I)=R$. It concludes that $I=0$ which is a contradiction. We infer that $I$ is not a direct summand of $R$ and so $R$ is not semisimple.\\
(ii) Suppose that $I$ is a nonzero ideal of $R$, then $I$ is a direct summand of $R$ since $R$ is semisimple. By (i), $I$ is not a sa-small ideal of $R$ and this is contradiction because $R$ is a sa-small hollow ring. It concludes that $R$ has no any nonzero ideal and so $R$ is a simple ring.\\
(iii) Assume that $x\in {\rm Z}(R)$, then there exists an element $0\neq y\in R$ such that $xy=0$. Let $R=Rx+Ry$, then 
$Ry\subseteq {\rm Ann}(Rx)\ll R$. This implies that $Ry \ll R$ hence $Rx=R$ and so $x\in {\rm U}(R)$ which is contradiction.\\
Now let $e$ be an idempotent element of $R$. Since $R(1-e)+Re=R$ therefore $R(1-e)\subseteq {\rm Ann}(Re)\ll R$. This implies that $R(1-e)\ll R$ and so $Re=R$. It concludes that $e=1$.\\
(iv) The proof follows from the fact that since $R$ is a von Neumann regular ring hence every finitely generated ideal $I$ of $R$ is a direct summand of $R$ such that $I$ is generated by
an idempotent. Suppose that $R=I\oplus J$, then $IJ=0$ and by (i) $I$ can not be a sa-small ideal. \\
(v) Take $N=0$, then $0+M=M$ and so ${\rm Ann}(M)=0$ is small in $R$. Now sssume that $N$ is a nonzero submodule of $M$, then $N=IM$ for some nonzero ideal $I$ of $R$. Let $r\in {\rm Ann}(N)$, then $r(IM)=0$ and so $rI=0$. It concludes that $r=0$ since $R$ is an integral domain and so ${\rm Ann}(N)=0$ which is a small ideal of $R$.
\end{proof}
\begin{prop}\label{pro2.6}
Let $M$ be an $R$-module with $N\leq K\leq M$ and let $T$ be an arbitrary submodule of $M$. The following statments are true.
\begin{itemize}
\item[(i)] Let $M$ be a strong comultiplication module and $N\ll ^{sa}M$. Then for every nonzero submodule $L$ of $M$ with $N+L=M$, 
$L\leq _{e}M$. 
\item[(ii)] If $K\ll ^{sa}_{T}M$, then $N\ll ^{sa}_{T}M$.
\item[(iii)] Assume $\{N_{\lambda}\}_{\lambda \in \Lambda}$ be a family of submodules of $M$. If $N_{t}\ll^{sa}_{T}M$ for some $t\in \Lambda $, then  $\cap _{\lambda \in \Lambda}N_{\lambda} \ll ^{sa}_{T}M$.
\item[(iv)] Let $T\subseteq K$ and  $N\ll ^{sa}_{T}K$, then $N\ll ^{sa}_{T}M$.
\item[(v)] Suppose that $M$ is a multiplication module. If $N\ll ^{sa}M$ (resp. $N\ll _{T}^{sa}M$), then $(N:M)\ll ^{sa}R$ (resp. $(N:M)\ll_{(T:M)}^{sa}R)$. The converse is true if $M$ is also a finitely generated faithful module. Furthermore, in this case, ${\rm J}^{sa}_{T}(M)={\rm J}^{sa}_{(T:M)}(R)M$.
\item[(vi)] Assume that $M$ and $M'$ are $R$-modules and $f:M\rightarrow M'$ is an $R$-epimorphism. If $N'\ll _{T'}^{sa}M'$ for some submodule $T'$ of $M'$, then $f^{-1}(N')\ll ^{sa}_{f^{-1}(T')}M$.
\item[(vii)] Assume that $T\supsetneq N$ is an arbitrary submodule of $M$. If $K/N\ll _{T/N}^{sa}M/N$, then $K\ll _{T}^{sa}M$ and $N\ll _{T}^{sa}M$. 
\item[(viii)] Let $M$ be a Noetherian $R$-module. If $S$ is a m.c.s. of $R$ and 
$S^{-1}N$ is an $S^{-1}T$-sa-small submodule of $S^{-1}R$-module $S^{-1}M$, then 
 $N$ is a $T$-sa-small submodule of $M$. 
\end{itemize}
\end{prop}
\begin{proof}
(i) Since $N\ll ^{sa}M$ and $N+L=M$ hence ${\rm Ann}(L)\ll R$. By virtue of \cite[Theorem 2.5]{WL}, $L=(0:_{M}{\rm Ann}_{R}(L))$ is an essential submodule of $M$.\\
(ii) Assume that $T\subseteq N+X$ for some submodule $X$ of $M$. Therefore $T\subseteq K+X$ and since $K\ll ^{sa}_{T}M$ hence ${\rm Ann}(X)\ll (T:M)$ and the proof is complete.  \\
(iii) It is clear by (ii).\\
(iv) Let $L\leq M$ and $T\subseteq N+L$, then by the modular law
$T\subseteq (N+L)\cap K=N+(L\cap K)$.
Since $N\ll ^{sa}_{T}K$, hence ${\rm Ann}(L)\subseteq {\rm Ann}(L\cap K)\ll (T:_{R}M)$ and this implies that ${\rm Ann}(L)\ll (T:_{R}M)$ and the proof is complete.\\
(v) ($\Rightarrow $) Let  $(N:M)+J=R$ for some ideal $J$ of $R$, then $(N:M)M+JM=M$. Since $M$ is multiplication hence 
$N+JM=M$ and so $\Ann (JM)\ll R$. From $\Ann (J)\subseteq \Ann (JM)$ we infer that $\Ann (J)\ll R$ and so $(N:M)\ll ^{sa}R$.\\
($\Leftarrow $) Suppose  that $(N:M)\ll ^{sa}R$ and $N+K=M$ for some submodule $K$ of $M$. Thus 
$(N:M)M+(K:M)M=RM$ and since $M$ is a finitely generated faithful multiplication module hence $M$ is a cancellation module and so $(N:M)+(K:M)=R$. It concludes that 
$\Ann (K:M)\ll R$. By hypothesis, since $M$ is a faithful module hence $\Ann (K:M)=\Ann (K)\ll R$ and the proof is complete. 

Now assume that $N\ll _{T}^{sa}M$ and $(T:M)\subseteq (N:M)+J$ for some ideal $J$ of $R$. We show that 
${\rm Ann}(J)\ll ((T:_{R}M):_{R}R)=(T:_{R}M)$. Since $M$ is a multiplication module hence
$T=(T:M)M\subseteq (N:M)M+JM=N+JM$.
Now the hypothesis $N\ll _{T}^{sa}M$ implies that ${\rm Ann}(J)\subseteq {\rm Ann}(JM)\ll (T:_{R}M)$ and so ${\rm Ann}(J)\ll (T:_{R}M)$ as we needed. Conversely, let $(N:M)\ll_{(T:M)}^{sa}R$ and $T\subseteq N+K$ for some submodules $N, K$ of $M$. Since $M$ is multiplication there exist ideals $I, J$ of $M$ such that $N=IM=(N:M)M$ and $K=JM=(K:M)M$.
Therefore $T=(T:M)M\subseteq (N:M)M+JM=((N:M)+J)M$ and so $(T:M)\subseteq (N:M)+J$, because $M$ is a cancellation module. From the hypothesis $(N:M)\ll_{(T:M)}^{sa}R$, we infer that ${\rm Ann}(J)\ll (T:M)$. Since $M$ is faithful hence ${\rm Ann}(J)={\rm Ann}(JM)\ll (T:M)$ and the proof is complete. For the second part we note that
\begin{align*}
{\rm J}^{sa}_{T}(M)&:=\sum _{N\ll ^{sa}_{T}M}N=\sum _{N\ll ^{sa}_{T}M}(N:M)M\\
&:=\left(\sum _{(N:M)\ll ^{sa}_{(T:M)}R}(N:M)\right) M={\rm J}^{sa}_{(T:M)}(R)M.
\end{align*}
(vi) Suppose that $f^{-1}(T')\subseteq f^{-1}(N')+L$ for some submodule $L$ of $M$. Then since $f$ is an $R$-epimorphism hence
\begin{equation*}
T'= f(f^{-1}(T'))\subseteq f(f^{-1}(N')+L)\subseteq N'+f(L). 
\end{equation*}
Since $N'\ll _{T'}^{sa}M'$ hence ${\rm Ann}(f(L))\ll (T':M')$ and therefore 
\[{\rm Ann}(L)\subseteq {\rm Ann}(f(L))\ll (T':_{R}M')=(f^{-1}(T'):_{R}M).\]
It implies that ${\rm Ann}(L)\ll (f^{-1}(T'):_{R}M)$ and the proof is complete.\\
(vii) Assume that $K/N\ll _{T/N}^{sa}M/N$ and also $T\subseteq K+L$ for some submodule $L$ of $M$. Then $T/N\subseteq (K+L)/N=K/N+L/N$ implies that ${\rm Ann}(L/N)\ll (T/N:M/N)=(T:M)$. Since ${\rm Ann}(L)\subseteq {\rm Ann}(L/N)\ll (T:M)$ hence ${\rm Ann}(L)\ll (T:M)$. It conclude that $K\leq _{T}^{sa}M$ and by (ii), $N\leq _{T}^{sa}M$.\\
(viii) We recall that if $X$ is a finitely generated submodule of $M$, then $S^{-1}(0:_{R}X)=(S^{-1}0:_{S^{-1}R}S^{-1}X)$.
Suppose that  $T\subseteq N+X$  we show that ${\rm Ann}_{R}(X)\ll (T:_{R}M)$. Then we have
\[S^{-1}T\subseteq S^{-1}(N+X)=S^{-1}N+S^{-1}X.\]
By hypothesis
${\rm Ann}_{S^{-1}R}(S^{-1}X)\ll (S^{-1}T:_{S^{-1}R}S^{-1}M)$. Therefore   
\begin{align*}
{\rm Ann}_{S^{-1}R}(S^{-1}X):&=(S^{-1}0:_{S^{-1}R}S^{-1}X)=S^{-1}(0:_{R}X)\\
\ll & S^{-1}(T:_{R}M)=(S^{-1}T:_{S^{-1}R}S^{-1}M).
\end{align*}
This implies that ${\rm Ann}_{R}(X)\ll (T:_{R}M)$ and the proof is complete.
\end{proof}
\begin{cor}
Let $M$ be a faithful finitely generated multiplication $R$-module and let $T$ be an arbitrary submodule of $M$. Then $M$ is a $T$-sa-hollow module if and only if $R$ is a $(T:M)$-sa-hollow ring.
\end{cor}
\begin{proof}
The proof is straightforward by Proposition \ref{pro2.6} (v). 
\end{proof}
\begin{cor}
Let $(R, \mathfrak{m})$ be a local ring and let  $A$ be an arbitrary ideal of $R$. If $\mathfrak{m}\in {\rm S}^{sa}_{A}(R)$, then $I\in {\rm S}^{sa}_{A}(R)$ for every ideal $I$ of $R$. 
\end{cor}
\begin{proof}
The proof is straightforward by Proposition \ref{pro2.6} (ii). 
\end{proof}
\begin{thm}
Let $f:M\rightarrow M'$ be an $R$-epimorphism and let $T'$ be a submodule of $M'$. If $M'$ is a $T'$-sa-hollow module, then 
$M$ is an $f^{-1}(T')$-sa-hollow module.
\end{thm} 
\begin{proof}
Assume that $K$ is a submodule of $M$. Then $f(K)\leq _{T'}^{sa}M'$ since $M'$ is a $T'$-sa-hollow module. By Proposition \ref{pro2.6} (vi), $f^{-1}(f(K))$ is an $f^{-1}(T')$-sa-small submodule of $M$. Since $K\subseteq f^{-1}(f(K))$, so by Proposition \ref{pro2.6} (ii), $K$ is also an $f^{-1}(T')$-sa-small submodule of $M$.
\end{proof}

The following example shows that the converse of Proposition \ref{pro2.6} (vii), is not true.
\begin{exam}
Consider the $\mathbb{Z}$-module $M=\mathbb{Z}$. Take $T=2\mathbb{Z}$, $K=4\mathbb{Z}$ and $N=8\mathbb{Z}$. Then $8\mathbb{Z}$ is an $2\mathbb{Z}$-sa-small submodule of $\mathbb{Z}$, because if $2\mathbb{Z}\subseteq 8\mathbb{Z}+k\mathbb{Z}$ for some submodule $k\mathbb{Z}$ of $\mathbb{Z}$, then either $(k, 8)=2$ or $(k, 8)=1$. In any case, $0={\rm Ann}(k\mathbb{Z})\ll (2\mathbb{Z}:_{\mathbb{Z}}\mathbb{Z})=2\mathbb{Z}$, but $K/N=4\mathbb{Z}/8\mathbb{Z}$ is not $2\mathbb{Z}/8\mathbb{Z}$-sa-small submodule of $\mathbb{Z}/8\mathbb{Z}$, because if $2\mathbb{Z}/8\mathbb{Z}\subseteq 4\mathbb{Z}/8\mathbb{Z}+k\mathbb{Z}/8\mathbb{Z}$ for some submodule $k\mathbb{Z}/8\mathbb{Z}$ of $\mathbb{Z}/8\mathbb{Z}$, then $2\mathbb{Z}\subseteq 4\mathbb{Z}+k\mathbb{Z}=(4, k)\mathbb{Z}$. If $k=2$, then ${\rm Ann}(2\mathbb{Z}/8\mathbb{Z})=4\mathbb{Z}$ is not small in $(2\mathbb{Z}/8\mathbb{Z}:_{\mathbb{Z}}\mathbb{Z}/8\mathbb{Z})=2\mathbb{Z}$. If $k=1$, then ${\rm Ann}(\mathbb{Z}/8\mathbb{Z})=8\mathbb{Z}$ is not small in $(2\mathbb{Z}/8\mathbb{Z}:_{\mathbb{Z}}\mathbb{Z}/8\mathbb{Z})=2\mathbb{Z}$.
\end{exam}
We recall that for an ideal $I$ of a ring $R$ the {\it radical} of $I$ is defined by $\rad (I)=\{x\in R\,\mid \, x^{n}\in I\,\, {\rm for}\,\, {\rm some}\, n\in \N\}$.  Let $N$ be a proper submodule of $M$. Then, the {\it prime radical} of
$N$, denoted by $\rad (N)$ is defined to be the intersection of all prime submodules of $M$ containing $N$, and in case $N$ is not contained in any prime submodule then $\rad (N)$ is defined to be $M$.
\begin{lem}\label{L3.16}
If $I\ll ^{sa}R$, then $\rad (I)\ll ^{sa}R$.
\end{lem}
\begin{proof}
Let $\rad (I)+J=R$ for some ideal $J$ of $R$. Since $\rad (I)+J\subseteq \rad (I)+\rad (J)$, so $\rad (I)+\rad (J)=R$. This implies that $\rad (I+J)=R$ and so $I+J=R$. Hence $\Ann (J)\ll R$ since $I\ll ^{sa}R$ and so $\rad (I)\ll ^{sa}R$.
\end{proof}
\begin{prop}
Let $M$ be a finitely generated faithful multiplication $R$-module. If $N\ll ^{sa}M$, then $\rad (N)\ll ^{sa}M$.
\end{prop}
\begin{proof}
By virtue of \cite[Theorem 4]{MM}, $\rad (N)=\rad (N:M)\,M$ and so
$(\rad (N):M)=\rad (N:M)$. Since $N\ll ^{sa}M$, then by Proposition \ref{pro2.6} (v), $(N:M)\ll ^{sa}R$ and by Lemma \ref{L3.16}, $\rad (N:M)\ll ^{sa}R$. It concludes that $(\rad (N):M)\ll ^{sa}R$. Again using Proposition \ref{pro2.6} (v)
we find that $\rad (N)\ll ^{sa}M$.
\end{proof}
\begin{thm}
Let $N$ be a nonzero sa-small submodule of $M$ and $K\leq M$ with $(N:K)+(K:N)=R$, then $N\cap K\neq 0$. 
\end{thm}
\begin{proof}
Suppose that $N\cap K=0$, then $(N:K)+(K:N)={\rm Ann}(K)+{\rm Ann}(N)=R$. Since $N\in {\rm S}^{sa}(M)$ hence 
${\rm Ann}(K)\ll R$ and so ${\rm Ann}(N)=R$ which is contradiction. 
\end{proof}
\begin{thm}
Let $K, H$ be submodules of $M$.
\begin{itemize}
\item[(i)] If $K+H\ll ^{sa}_{T}M$, then $K\ll ^{sa}_{T}M$ and $H\ll ^{sa}_{T}M$.
\item[(ii)] If $M$ is a prime module and $K\ll ^{sa}_{T}M$ and $K+H\neq M$, then $K+H\ll ^{sa}_{T}M$.
\end{itemize}
\end{thm}
\begin{proof}
(i) It is clear by Proposition \ref{pro2.6} (ii).\\
(ii) Let $T\subseteq K+H+X$ for some submodule $X$ of $M$. Since $K\ll ^{sa}_{T}M$ hence ${\rm Ann}(X)={\rm Ann}(H+X)\ll (T:M)$.
\end{proof}
\begin{thm}
Let $R$ be a semisimple hollow ring and let $M_{1}, M_{2}$ be $R$-modules. Suppose that $N_{1}\ll ^{sa}_{T_{1}}M_{1}$ and $N_{2}\ll ^{sa}_{T_{2}}M_{2}$ for submodules $T_{1}\leq M_{1}$ and $T_{2}\leq M_{2}$, then
$N_{1}\oplus N_{2}\ll ^{sa}_{T_{1}\oplus T_{2}}M_{1}\oplus M_{2}$.
\end{thm}
\begin{proof}
Suppose that 
\[T_{1}\oplus T_{2}\subseteq (N_{1}\oplus N_{2})+(L_{1}\oplus L_{2})=(N_{1}+L_{1})\oplus (N_{2}+L_{2}).\]
Then $T_{1}\subseteq N_{1}+L_{1}$ and $T_{2}\subseteq N_{2}+L_{2}$ hence ${\rm Ann}(L_{1})\ll (T_{1}:M_{1})$ and 
${\rm Ann}(L_{2})\ll (T_{2}:M_{2})$. By  \cite[Lemma 2]{Leo}, if $S\subseteq E\subseteq F$ and $S\ll F$ such that $E$ is a direct summand of $F$, then $S\ll E$. Take $S={\rm Ann}(L_{1})\cap {\rm Ann}(L_{2})$ and  $F= (T_{1}:M_{1})\cap (T_{2}:M_{2})$, then 
\begin{align*}
{\rm Ann}(L_{1}\oplus L_{2})&={\rm Ann}(L_{1})\cap {\rm Ann}(L_{2})\ll (T_{1}:M_{1})\cap (T_{2}:M_{2})\\
&=(T_{1}\oplus T_{2}:M_{1}\oplus M_{2}).
\end{align*} 
\end{proof}
\begin{thm}
Let $f:M\rightarrow N$ be a monomorphism and let $T$ be an arbitrary submodule of $M$. If $K\ll  ^{sa}_{T}M$, then $f(K)\ll _{f(T)}^{sa}f(M)$.
\end{thm}
\begin{proof}
Assume that $f(T)\subseteq f(K)+L$ for some submodule $L$ of $f(M)$. We show that ${\rm Ann}(L)\ll (f(T):_{R}f(M))$. We have
\[T=f^{-1}(f(T))\subseteq f^{-1}(f(K)+L)=K+f^{-1}(L)\leq M.\]
Then ${\rm Ann}(L)\subseteq {\rm Ann}(f^{-1}(L))\ll (T:_{R}M)=(f(T):_{R}f(M))$.
\end{proof}
\begin{thm}
Let $f:N\rightarrow K$ be a monomorphism and let $T$ be an arbitrary submodule of $N$. If $g:K\rightarrow M$ is an $f(T)$-sa-small epimorphism, then $g\circ f:N\rightarrow M$ is also a $T$-sa-small epimorphism.
\end{thm}
\begin{proof}
Assume that $T\subseteq {\rm Ker} (g\circ f)+X$ for some submodule $X$ of $N$. We show that ${\rm Ann}(X)\ll (T:N)$. Since ${\rm Ker}(g\circ f)=f^{-1}({\rm Ker}g)$ hence 
\[f(T)\subseteq f({\rm Ker} (g\circ f))+f(X)\subseteq {\rm Ker} g+f(X).\]
Since ${\rm Ker} g\ll ^{sa}_{f(T)}K$, hence 
\[{\rm Ann}(X)\subseteq {\rm Ann}(f(X))\ll (f(T):_{R}K).\]
We have $(f(T):_{R}K)\subseteq (T:_{R}N)$, because if $r\in (f(T):_{R}K)$ and $x\in N$, then $rf(x)=f(rx)\in f(T)$. Since $f$ is monomorphism hence $rx\in f^{-1}(f(T))=T$.
It concludes that $r\in (T:N)$. Therefore ${\rm Ann}(X)\subseteq {\rm Ann}(f(X))\ll (T:_{R}N)$ and the proof is complete.
\end{proof}
We recall that an $R$-module $F$ is called {\it flat} if whenever $N\rightarrow K\rightarrow L$ is an exact sequence of $R$-modules, then $F\otimes N\rightarrow F\otimes K\rightarrow F\otimes L$ is an exact sequence as well.
An $R$-module $F$ is called  {\it faithfully flat}, whenever $N\rightarrow K\rightarrow L$ is an exact sequence of $R$-modules if and only if $F\otimes N\rightarrow F\otimes K\rightarrow F\otimes L$ is an exact sequence.
\begin{thm}
Let $F$ be a faithfully flat $R$-module and let $M$ be an $R$-module. Assume that $N\leq M$ and $T$ is an arbitrary submodule of $M$. Then the following statements hold.
\begin{itemize}
\item[(i)] $N$ is a sa-small submodule of $M$ if and only if $F\otimes N$ is a sa-small submodule of $F\otimes M$. 
\item[(ii)] $N$ is a $T$-sa-small submodule of $M$ if and only if $F\otimes N$ is a $F\otimes T$-sa-small submodule of 
$F\otimes M$. 
\end{itemize}
\end{thm}
\begin{proof}
(i) ($\Rightarrow $) Let $N\leq ^{sa}M$ and $F\otimes N+F\otimes K=F\otimes M$ for some submodule $F\otimes K$ of $F\otimes M$. Then $F\otimes (N+K)=F\otimes N+F\otimes K=F\otimes M$. Thus $0\rightarrow F\otimes (N+K)\rightarrow 
F\otimes M\rightarrow 0$ is an exact sequence. Since $F$ is faithfully flat, so $0\rightarrow N+K\rightarrow M\rightarrow 0$ is an exact sequence. Therefore $N+K=M$ and this implies that ${\rm Ann}(K)\ll R$ since $N\leq ^{sa}M$. We have 
${\rm Ann}(K)={\rm Ann}(F\otimes K)$, because if $r\in {\rm Ann}(K)$, then $rK=0$. Thus $r(F\otimes K)=F\otimes rK=0$ and so $r\in {\rm Ann}(F\otimes K)$. If $r\in {\rm Ann}(F\otimes K)$, then $0\rightarrow F\otimes rK\rightarrow 0$ is an exact sequence and so $0\rightarrow rK\rightarrow 0$ is exact since $F$ is a faithfully flat $R$-module. Therefore $rK=0$ and so $r\in {\rm Ann}(K)$. It concludes that ${\rm Ann}(F\otimes K)\ll R$, so $F\otimes N\leq ^{sa}F\otimes M$.\\
($\Leftarrow $) Let $N+K=M$ for some submodule $K$ of $M$. Then
\[F\otimes (N+K)=F\otimes N+F\otimes K=F\otimes M.\]
Hence ${\rm Ann}(K)={\rm Ann}(F\otimes K)\ll R$. Therefore $N\leq ^{sa}M$, as needed. \\
\\
(ii) ($\Rightarrow $) Let $F\otimes T\subseteq F\otimes N+F\otimes K$ for some submodule $F\otimes K$ of $F\otimes M$. Thus
$F\otimes T\subseteq F\otimes (N+K)$ and so $0\rightarrow F\otimes T\rightarrow F\otimes (N+K)$ is exact. Therefore 
$0\rightarrow T\rightarrow N+K$ is also exact since $F$ is a faithfully flat $R$-module. This implies that $T\subseteq N+K$ and since $N\ll ^{sa}_{T}M$ hence ${\rm Ann}(K)\ll (T:_{R}M)$. It is easy to see that $(T:_{R}M)=(F\otimes T:_{R}F\otimes M)$ since $F$ is faithfully flat. 
Hence ${\rm Ann}(F\otimes K)\ll (F\otimes T:_{R}F\otimes M)$  and so $F\otimes N\ll _{F\otimes T}^{sa}F\otimes  M$.\\
($\Leftarrow $) Let $T\subseteq N+K$. Then $F\otimes T\subseteq F\otimes (N+K)=F\otimes N+F\otimes K$. Now since $F\otimes N\ll _{F\otimes T}^{sa}F\otimes  M$ hence ${\rm Ann}(F\otimes K)\ll (F\otimes T:_{R}F\otimes M)$. It concludes that ${\rm Ann}(K)\ll (T:_{R}M)$ and so $N\ll ^{sa}_{T}M$.
\end{proof}

\bigskip
\bigskip

{\footnotesize {\bf Saeed Rajaee}\; \\ {Department of Mathematics}, { Payame Noor University, P.O.Box 19395-3697}, {Tehran, Iran.}\\
{\tt Email: saeed\_ rajaee@pnu.ac.ir}\\
{\footnotesize {\bf Farkhondeh Farzalipour}\; \\ {Department of Mathematics}, {Payame Noor University, P.O.Box 19395-3697}, {Tehran, Iran.}\\
{\tt Email: f\_ farzalipour@pnu.ac.ir}\\
{\footnotesize {\bf Marzieh Poyan}\; \\ {Department of Mathematics, PhD student of Mathematical Sciences}, {Payame Noor University, P.O.Box 19395-3697}, {Tehran, Iran.}\\
{\tt Email: r\_ poyan@pnu.ac.ir}

\end{document}